\documentclass{amsart}
\usepackage{amssymb, latexsym}

\DeclareMathOperator{\Tr}{Tr}
\DeclareMathOperator{\Alt}{Alt}
\DeclareMathOperator{\Bil}{Bil}

\theoremstyle{plain}
\newtheorem{theorem}{Theorem}
\newtheorem{corollary}{Corollary}
\newtheorem{lemma}{Lemma}

\theoremstyle{definition}
\newtheorem{definition}{Definition}
\begin{document}
\title[Galois extensions and subspaces of alternating bilinear forms]
{Galois extensions and subspaces of alternating \\ 
bilinear forms with special rank properties} 

\author{Rod Gow and Rachel Quinlan}
\address{School of Mathematical Sciences\\
University College\\
Belfield, Dublin 4\\
Ireland}
\address{Department of Mathematics\\
National University of Ireland\\
Galway\\
Ireland}
\email{rod.gow@ucd.ie, rachel.quinlan@nuigalway.ie}
\begin{abstract} 
Let $K$ be a field admitting a cyclic Galois extension of degree $n$. The main result of this paper is a decomposition theorem
for the space of alternating bilinear forms defined on a vector space of odd dimension $n$ over $K$. We show that this space of forms is the direct sum of $(n-1)/2$ subspaces, each of dimension $n$, and the non-zero elements in each subspace have constant rank defined in terms of the orders of the Galois automorphisms. Furthermore, if ordered correctly, for each integer $k$ lying between 1 and $(n-1)/2$, the rank of any non-zero element in the sum of the first $k$ subspaces is at most $n-2k+1$. Slightly less sharp similar results hold for cyclic extensions of even degree.
\end{abstract}
\maketitle
\section{Introduction}
\noindent  Let $K$ be a field and let $L$ be a Galois extension of $K$ of finite degree $n$. 
Let $G$ denote the Galois group of $L$ over $K$. We may consider $L$ to be a vector space of dimension $n$ over $K$ and it serves as a model for any such $n$-dimensional vector space. However,
the Galois nature of the extension enriches the vector space structure and enables us to carry out constructions which we could not achieve without the additional field-theoretic apparatus. In this paper, we intend to investigate alternating bilinear forms defined on $L\times L$, and taking values in $K$. We will be particularly interested in subspaces of such forms and their properties with respect to rank.

Let $\Tr=\Tr^L_K$ denote the trace form $L\to K$, defined by
\[
\Tr(x)=\sum_{\sigma\in G} \sigma(x).
\]
As is well known, this form is not identically zero and it has the Galois invariance property
\[
\Tr(\tau(x))=\Tr(x)
\]
for all $x\in L$ and all $\tau\in G$. It is our main tool for making and investigating  subspaces of alternating bilinear forms.

\section{Galois extensions and alternating bilinear forms}

\noindent We now begin the study of the main topic of this paper. Let $K^*$ and $L^*$ denote the multiplicative groups of $K$ and $L$, respectively.  Let $b$ be an element of $L^*$ and $\sigma$
a non-identity element of $G$. Define 
\[
f=f_{b,\,\sigma}: L\times L \to K
\]
by
$$
f(x,y)=\Tr (b(x\sigma(y)-\sigma(x) y))
$$
for all $x$ and $y$ in $L$. It is straightforward to see that $f$ is $K$-bilinear considering $L$ as a vector space over $K$ and moreover
\[
f(x,x)=\Tr (b(x\sigma(x)-\sigma(x) x))=\Tr(0)=0.
\]
Thus $f$ is an alternating bilinear form. The corresponding form
when $b=0$ is  the zero form. We note also
the following basic property of $f$, which we state as our first lemma.
\begin{lemma} \label{invariance} Let $f=f_{b,\,\sigma}$ be the alternating bilinear form defined above.
Then we have
\[
f(x,y)=\Tr((\sigma^{-1}(bx)-b\sigma(x))y)
\]
for all $x$ and $y$ in $L$.
\end{lemma}

\begin{proof}
It follows from the Galois invariance of the trace form that 
\[
\Tr((bx)\sigma(y))=\Tr(\sigma^{-1}(bx)y)
\]
and the lemma is an immediate consequence of this observation.
\end{proof}

The non-degeneracy of the trace form implies the following important consequence of Lemma \ref{invariance}. 

\begin{corollary} \label{radical} An element $x$ is in the radical of $f_{b,\,\sigma}$ if and only if
\[
\sigma^{-1}(bx)=b\sigma(x).
\]
\end{corollary}

We proceed next to calculate the rank of the alternating form $f$.

\begin{lemma} \label{rank} Let $f=f_{b,\,\sigma}$ be the alternating bilinear form defined above, with
$b\ne 0$,  and let $F$ be the fixed field of the automorphism $\sigma^2$. Then if $\sigma(b) b^{-1}$ is expressible in the form $\sigma^2(c)c^{-1}$ for some $c\in L^*$, we have
\[
\mbox{\rm rank}\,f=n-[F:K]=n-\frac{n}{[L:F]}.
\]
Otherwise, $f$ has rank $n$ and is thus non-degenerate.
\end{lemma}

\begin{proof} Let $R$ denote the radical of $f$ and let $x$ be an element of $R$. We have
\[
\sigma^{-1}(bx)=b\sigma(x)
\]
by Corollary \ref{radical}. Applying $\sigma$ to this equality, we obtain 
\[
bx=\sigma(b) \sigma^2(x).
\]
It follows that if $x$ is non-zero, then, taking $x=c^{-1}$, we must have 
\[
\sigma(b) b^{-1}=\sigma^2(c)c^{-1},
\]
as claimed. On the other hand, if $\sigma(b)b^{-1}$ is not expressible in the stated form, there can be no non-zero solution for $x$ and hence $R=0$, which implies that $f$ has rank $n$. Finally, suppose that
\[
\sigma(b) b^{-1}=\sigma^2(c)c^{-1}
\]
for some non-zero $c$. Then any non-zero $x\in R$ satisfies
\[
x^{-1}c\sigma^2(x)\sigma^2(c)^{-1}=1,
\]
from which it follows that $x\in c F^*$. We deduce that
\[
\dim R=[F:K]
\]
and hence the rank of $f$ is
\[
n-[F:K]=n-\frac{n}{[L:F]}.
\]
\end{proof}

We show next that there is a simple formula for the rank of $f_{b,\,\sigma}$ when $\sigma$ has odd multiplicative order. 

\begin{lemma} \label{oddorderrank} Suppose that the automorphism $\sigma$ has odd multiplicative order $2r+1>1$, say. Then, if $b\ne 0$,  the rank of the alternating bilinear form $f=f_{b,\sigma}$ equals
\[
n-\frac{n}{2r+1}.
\]
\end{lemma}

\begin{proof} 

Under the given supposition, $\sigma$ and $\sigma^2$ generate the same cyclic subgroup, $H$ say, of $G$, and thus the
fixed field $F$ of $\sigma^2$ equals the fixed field of $\sigma$. Let $T_1:L\to F$ be the trace form
defined by
\[
T_1(x)=\sum_{i=1}^{2r+1} \sigma^i(x)
\]
for $x\in L$. Let $\tau_1$, \dots, $\tau_k$ be a complete set of coset representatives of the subgroup $H$ in $G$ and let $T_2:F\to K$ be the $K$-linear form defined by
\[
T_2(z)=\sum_{i=1}^k \tau_i(z)
\]
for $z\in F$. Then we have
\[
\Tr(x)=T_2(T_1(x))
\]
for all $x\in L$. 

Considering $L$ as a vector space over $F$, we may define an alternating $F$-bilinear form
\[
g: L\times L\to F
\]
by
\[
g(x,y)=T_1(b(x\sigma(y)-\sigma(x) y))
\]
for all $x$ and $y$ in $L$. Then we have
\[
f(x,y)=T_2(g(x,y)).
\]
Now as $L$ has odd dimension $2r+1$ over $F$, we know that any alternating bilinear form from
$L\times L$ to $F$ has a non-zero radical, \cite{L} Theorem 8.1. Thus there exists $x\ne 0$ in $L$ with
\[
g(x,y)=0
\]
for all $y\in L$. But this implies that $x$ is also in the radical of $f$, by the equation above relating $f$ and $g$,  and hence $f$ is degenerate.
Lemma \ref{rank} therefore yields that $f$ has rank
\[
n-\frac{n}{[L:F]}=n-\frac{n}{[2r+1]},
\]
as required.

\end{proof}
 
 An alternative approach to proving Lemma \ref{oddorderrank} is to use Hilbert's Theorem 90. 
 
 As we shall see in the next result, the situation described in Lemma \ref{oddorderrank}  when $\sigma$ has odd order is less clear cut when $\sigma$ has even order.

\begin{lemma} \label{nonexistence} Suppose that the automorphism $\sigma$ has even multiplicative order $2r$, say. Then there exist elements $b\in L^*$ such that the equation
\[
\sigma(b) b^{-1}\ne \sigma^2(c)c^{-1}
\]
has no solution for  $c\in L^*$.

\end{lemma}

\begin{proof} Suppose first that $K$ is finite. Then as $L$ is an extension of $F$ of even degree $2r$, we
have
\[
|F|=q,\quad |L|=q^{2r},
\]
for some prime power $q$, and we can take $\sigma$ to be given by the Frobenius mapping $x\to x^q$.
Then a straightforward calculation using the cyclicity of $L^*$ implies that the number of elements
of the form $b^\sigma b^{-1}$ is
\[
\frac{q^{2r}-1}{q-1},
\]
whereas the number of elements of the form  $\sigma^2(c)c^{-1}$ is
\[
\frac{q^{2r}-1}{q^2-1}.
\]
Since the latter number is smaller than the former, our claim follows in this case.

We suppose now that $K$ is infinite. Suppose also by way of contradiction that every element
$\sigma(b) b^{-1}$ is expressible as $ \sigma^2(c)c^{-1}$ for some $c$.
Then applying the powers of $\sigma$, we obtain
\[
b\sigma^2(b)\cdots \sigma^{2r-2}(b) = \sigma(b)\sigma^3(b)\cdots
\sigma^{2r-1}(b),
\]
an equality that holds for all $b\in L^*$. In particular, if we take any element $\alpha$ in $K$, and replace
$b$ above by $\alpha-b$, we also have
\[
(\alpha -b)\sigma^2(\alpha -b)\cdots \sigma^{2r-2}(\alpha -b) = 
\sigma(\alpha -b)\sigma^3(\alpha -b)\cdots \sigma^{2r-1}(\alpha-b).
\]
This in turn implies that
\[
(\alpha -b)(\alpha -\sigma^2(b))\cdots (\alpha -\sigma^{2r-2}(b)) = 
(\alpha -\sigma(b))(\alpha -\sigma^3(b))\dots (\alpha
-\sigma^{2r-1}(b)).
\]
We proceed to choose  an element $b$ for which the $2r$ Galois conjugates 
\[
b,\,\sigma(b),\ldots ,\,\sigma^{2r-1}(b)
\]
are all different, and define the polynomials $q_1$ and $q_2$ in the polynomial ring
$L[t]$ by 
\begin{eqnarray*}
q_1 & = & (t-b)(t-\sigma^2(b))\cdots (t-\sigma^{2r-2}(b)) \\
q_2 & = & (t-\sigma(b))(t-\sigma^3(b))\cdots (t-\sigma^{2r-1}(b)).
\end{eqnarray*}
These two polynomials have no roots in common. On the other hand, by what we have proved
above, for any element $\alpha$ of $K$,
\[
q_1(\alpha)=q_2(\alpha)
\]
This yields a contradiction, since it implies that two different monic polynomials agree on infinitely many different points. Thus our claim also holds when $K$ is infinite.
\end{proof}

\begin{corollary}  \label{evenorderrank} Suppose that the automorphism $\sigma$ has even multiplicative order $2r>2$, say. Then, as  $b$ runs over the elements in $L^*$, the alternating bilinear forms $f_{b,\sigma}$ have
rank either 
\[
n-\frac{n}{r}\, \mbox{ or }\, n. 
\]
Examples of each rank occur for suitable choices of $b$.
\end{corollary}

We point out here that, if in Corollary \ref{evenorderrank}, $\sigma$ has order 2, then the elements $f_{b,\,\sigma}$ are either 0 or else have rank $n$.

\begin{definition} Let $\Alt(L)$  
denote the set of all alternating bilinear forms $L\times L\to K$. 
\end{definition}

$\Alt(L)$ is naturally a vector space of dimension $n(n-1)/2$ over $K$.
Our application of  the trace form has constructed certain $n$-dimensional subspaces of Alt$(L)$, which
we formally define below.

\begin{definition} Let $\sigma$ be a non-identity element  of the Galois group $G$. We set 
\[
A^\sigma=\{\, f_{b,\,\sigma}:\, b\in L\,\}.
\]
\end{definition}

It is straightforward to see that $A^\sigma$ is a subspace of $\Alt(L)$, and $A^\sigma=A^{\sigma^{-1}}$.
We investigate properties of this
subspace next.

\begin{theorem} \label{subspacesrank} Let $\sigma$ be a non-identity element of the Galois group $G$ of $L$ over $K$. Then, unless $\sigma$ has order $2$, $\dim A^\sigma=n$. 
When $\sigma$ has odd order $2r+1$, the non-zero elements of $A^\sigma$ have constant rank 
\[
n-\frac{n}{2r+1}.
\]
When $\sigma$ has even order $2r>2$, the non-zero elements of $A^\sigma$ either have maximal rank $n$ or else have rank 
\[
n-\frac{n}{r}.
\]
$A^\sigma$ contains elements of both non-zero ranks. In the exceptional case that $\sigma$ has order $2$, $A^\sigma$ has dimension $n/2$.

\end{theorem}

\begin{proof}
Only the statements about  $\dim A^\sigma$  require explanation.
Now for $b$ in $L$, the mapping
\[
b\to f_{b,\,\sigma}
\]
is $K$-linear. Since $L$ is $n$-dimensional as a vector space over $K$, it suffices for us to decide 
when $f_{b,\,\sigma}$
is the zero form. Now Corollary \ref{radical} implies that if this form is 0, then
\[
bx=\sigma(b)\sigma^2(x)
\]
for all $x$ in $L$. Taking $x=1$, we obtain $b=\sigma(b)$. Hence, if $b$ is non-zero, we must have
\[
\sigma^2(x)=x
\]
for all $x$ and we deduce that $\sigma$ has order 2. Conversely, if $\sigma$ has order 2 and $b$ is fixed by $\sigma$, it is clear that $f_{b,\,\sigma}=0$. Since in this case the fixed field of $\sigma$
has degree $n/2$ over $K$, $\dim A^\sigma=n/2$ also.
\end{proof}

Our next intention is to show that when $[L:K]=n$ is odd, the space $\Alt(L)$ is the direct sum of $(n-1)/2$ $n$-dimensional subspaces of the form $A^\sigma$, where $\sigma\in G$. 

Suppose then that $n=2m+1$ is odd and enumerate the elements of the Galois group $G$ in the form
\[
1,\,\sigma_1,\,\sigma_1^{-1}, \ldots,\,\sigma_m,\,\sigma_m^{-1}.
\]
We shall correspondingly write $A^i$ in place of $A^{\sigma_i}$, so that
\[
A^i=\{\,f_{b,\,\sigma_i}:\,b\in L\,\}.
\]

Let $L^m$ denote the set of ordered $m$-tuples $(x_1, \ldots, x_m)$ of elements of $L$. $L^m$ is naturally a vector space of dimension $nm=n(n-1)/2$ over $K$ under the usual rules of addition of $m$-tuples and scalar multiplication.

%We next introduce some familiar notation. Given $\tau_1$, \dots, $\tau_k$ in the Galois group $G$ and elements $b_1$, \dots, $b_k$ in $L$, we define the function
%\[
%b_1\tau_1+\cdots +b_k\tau_k: \,L\to L
%\]
%by 
%\[
%(b_1\tau_1+\cdots +b_k\tau_k)(x)=b_1\tau_1(x)+\cdots +b_k\tau_k(x).
%\]
%We recall the following consequence of Artin's theorem on the linear independence of characters.
%Suppose that $\tau_1$, \dots, $\tau_k$ are distinct elements of $G$ and we have an equality
%\[
%b_1\tau_1+\cdots +b_k\tau_k=0.
%\]
%Then $b_1=\cdots=b_k=0$.

\begin{theorem} \label{isomorphism} Suppose that $n=[L:K]$ is odd, $n>1$,  and set $m=(n-1)/2$. Enumerate the elements of $G$
as
\[
\{\, 1,\,\sigma_1,\,\sigma_1^{-1}, \ldots,\,\sigma_m,\,\sigma_m^{-1}\,\}
\]
Define a $K$-linear mapping 
\[
\phi: L^m\to \Alt(L)
\]
by 
\[
\phi(b_1, \ldots, b_m)=\sum_{i=1}^m f_{b_i, \sigma_i}.
\]
Then $\phi$ is an isomorphism of $K$-vector spaces.
\end{theorem}

\begin{proof} As $L^m$ and $\Alt(L)$ have the same dimension over $K$, it suffices to show that if
\[
\phi(b_1,\ldots, b_m)=0,
\]
then 
\[
b_1=\cdots =b_m=0.
\]
Suppose therefore that 
\[
\phi(b_1, \ldots, b_m)(x,y)=\sum_{i=1}^m f_{b_i,\,\sigma_i}(x,y)=0
\]
for all $x$ and $y$. Using the linearity of the trace form, we obtain from  Corollary \ref{radical} that
\[
\sum_{i=1}^m \sigma_i^{-1}(b_ix)-b_i\sigma_i(x)=0
\]
for all $x\in L$. We therefore obtain
\[
\sum_{i=1}^m \sigma_i^{-1}(b_i) \sigma_i^{-1}-b_i\sigma_i=0.
\]
However, since the automorphisms appearing in the sum above are all distinct, Artin's theorem on the independence of characters, \cite{L} Theorem 4.1,  implies
that
\[
b_1=\cdots =b_m=0,
\]
as required.
\end{proof}

\begin{corollary} \label{odddirectsum} Suppose that $n=[L:K]$ is odd and set $m=(n-1)/2$. Then there is a direct sum decomposition
\[
\Alt(L)=A^1\oplus \cdots \oplus A^m,
\]
where each subspace $A^i$ has dimension $n$. If $A^i$ is defined by the automorphism $\sigma_i$ or its inverse, and if $\sigma_i$ has order $2r_i+1$, the non-zero elements of $A^i$ all have rank
\[
n-\frac{n}{2r_i+1}.
\]
\end{corollary}

We turn to establishing an analogue of Corollary \ref{odddirectsum} for Galois extensions of even degree. Theorem \ref{subspacesrank}  shows us that we must expect to use subspaces containing elements of two different ranks
when decomposing $\Alt(L)$ and that elements of order 2 (involutions) in the Galois group
lead to slightly different results.

\begin{theorem} \label{epimorphism} Suppose that $n=[L:K]$ is even and 
the Galois group $G$ of $L$ over $K$ contains exactly $k$ elements of order 2, $\tau_1$, \dots, $\tau_k$, say.  Let $F_i$ be the fixed field of $\tau_i$ for $1\le i\le k$. Enumerate the remaining non-identity elements of $G$
as
\[
\{\, \sigma_1,\,\sigma_1^{-1}, \ldots,\,\sigma_m,\,\sigma_m^{-1}\,\},
\]
where $1+k+2m=n$.
Define a $K$-linear mapping 
\[
\phi: L^{k+m}\to \Alt(L)
\]
by 
\[
\phi(b_1,\ldots,  b_k, c_1, \ldots, c_m)=\sum_{i=1}^kf_{b_i,\,\tau_i}+\sum_{j=1}^m f_{c_j, \sigma_j}.
\]
Then $\phi$ maps $L^{k+m}$ onto $\Alt(L)$ and the kernel of $\phi$, which has dimension $kn/2$,  consists of all elements
\[
(b_1, \ldots,  b_k, 0, \ldots, 0),
\]
where $b_i\in F_i$, $1\le i\le k$. 
\end{theorem}

\begin{proof} It is straightforward to see from the proofs of Theorems \ref{subspacesrank} and \ref{isomorphism}  that $\ker \phi$ consists of
elements of the stated form. Since each subfield $F_i$ satisfies $[F_i:K]=n/2$,  we have
\[
\dim \ker \phi=\frac{kn}{2}.
\]
Furthermore, since
\[
\dim L^{k+m}-\dim \ker \phi=(k+m)n-\frac{kn}{2}=\frac{n(n-1)}{2}=\dim \Alt(L),
\]
it is clear that $\phi$ is surjective.

\end{proof}

\begin{corollary} \label{evendirectsum} Assume the hypotheses and notation of Theorem \ref{epimorphism}. There is a direct sum decomposition
\[
\Alt(L)=B^1\oplus \cdots \oplus B^k\oplus  A^1\oplus \cdots\oplus A^m,
\]
where 
\[
\dim B^i=\frac{n}{2}, \, 1\le i\le k, \quad \dim A^j=n,\, 1\le j\le m.
\]
 $B^i$ is defined by the involution $\tau_i$ and all its non-zero elements have rank $n$, for
 $1\le i\le k$.   $A^j$ is defined by the automorphism $\sigma_j$
or its inverse, for $1\le j\le m$,  and the rank of each element in $A^j$ is given in accordance with the results of Lemma \ref{oddorderrank} and
Corollary \ref{evenorderrank}. 
\end{corollary}

\section{Cyclic extensions and rank properties}

\noindent We have decomposed the space $\Alt(L)$ of alternating bilinear forms into a direct sum of subspaces
$A^i$ which have special properties with respect to the ranks of their elements, particularly when $[L:K]$
is odd. We would also like to know if sums of certain of these subspaces also enjoy restricted rank properties. It is our intention to investigate this problem in the case that the Galois group $G$ is cyclic.

Suppose then that $L$ is a cyclic Galois extension of $K$ of degree $n$ and $\sigma$ generates
the Galois group $G$ of $L$ over $K$. Given elements $x_1$, \dots, $x_k$ of $L$, we seek a method
to decide when these elements are linearly dependent over $K$. The following result generalizes
a known criterion valid when $K$ is finite, \cite{Lidl}, Lemma 3.51. 

\begin{theorem} \label{lineardependence} Let $L$ be a cyclic Galois extension of $K$ of degree $n$ and suppose that $\sigma$ generates
the Galois group of $L$ over $K$. Let $k$ be an integer satisfying $1\le k\le n$ and let  $x_1$, \dots, $x_k$ be elements of  $L$. Then these elements are linearly dependent over $K$ if and and only if
\[
\det S=0,
\]
where $S$ is  the $k\times k$ matrix $S$ whose $i,j$-entry is 
\[
\sigma^{i-1}(x_j), \quad 1\le i, j\le k.
\]
\end{theorem}

\begin{proof} Suppose first that  $x_1$, \dots, $x_k$ are linearly dependent over $K$.
Then there are elements $a_1$, \dots, $a_k$ of $K$, not all 0, with
\[
a_1x_1+\cdots +a_kx_k=0.
\]
Applying the powers of $\sigma$ to this equation, we obtain
\[
a_1\sigma^{i-1}(x_1)+\cdots +a_k\sigma^{i-1}(x_k)=0
\]
for $1\le i\le k$. We thus have a homogeneous system of $k$ linear equations with square coefficient matrix $S
$ and this system has a non-trivial solution. It follows that $\det S=0$ (of course this argument does not depend on cyclicity). 

Suppose conversely that $\det S=0$. We prove that the elements are linearly dependent  by induction on $k$. 
Since the result is trivial when $k=1$, we may assume that
$k>1$. Furthermore, we may clearly assume that $x_1\ne 0$. Define elements $z_1$, $z_2$, \dots, $z_k$ of $L$ by
\[
z_1=1,\quad z_i=x_1^{-1}x_i
\]
for $2\le i\le k$. We now multiply all the entries of the $i$-th row of $S$ by $\sigma^{i-1}(x_1)^{-1}$
for $1\le i\le k$. Let $T$ be the matrix thus produced. We find that the $i,j$-entry of $T$ is
\[
\sigma^{i-1}(z_j)
\]
and thus, in particular, the first column of $T$ consists entirely of 1's. We also have $\det T=0$.

Next, we subtract successively row $k-1$ of $T$ from row $k$, row $k-2$ from row $k-1$, \dots,
row 1 from row 2. Let $U$ be the matrix obtained. Clearly, $\det U=0$ and all entries of the first
column of $U$ are 0, with the exception of the top entry, which is 1. Let $V$ be the $k-1\times k-1$ matrix
obtained by deleting the first row and column of $U$. Expansion of $\det U$ along the first row
shows that 
\[
\det V=\det U=0.
\]
Note also that the $i,j$-entry of $V$ is
\[
\sigma^i(z_{j+1})-\sigma^{i-1}(z_{j+1}), \quad 1\le i,j\le k-1.
\]

We now set
\[
y_j=\sigma(z_{j+1})-z_{j+1}, \quad 1\le j\le k-1
\]
and observe that the $i,j$-entry of $V$ is $\sigma^{i-1}(y_j)$. The induction hypothesis applies
to the $k-1$ elements $y_1$, \dots, $y_{k-1}$, since $\det V=0$,  and we see that
these elements are linearly dependent over $K$. Hence there  are elements $b_1$, \dots, $b_{k-1}$ of $K$, not all 0, with
\[
b_1y_1+\cdots +b_{k-1}y_{k-1}=0.
\]
Expressing the $y_i$ in terms of the $z_i$, this leads to
\[
b_1\sigma(z_2)+\cdots +b_{k-1}\sigma(z_k)=b_1z_2+\cdots +b_{k-1}z_k
\]
and we see that $\sigma$ fixes $b_1z_2+\cdots +b_{k-1}z_k$. Since the elements fixed by
$\sigma$ lie in $K$, we deduce that there is some element $b_k$, say, of $K$ with
\[
b_1z_2+\cdots +b_{k-1}z_k=b_k.
\]
If we now substitute $z_i=x_1^{-1}x_i$ for $2\le i\le k$, we find that we have obtained a non-trivial
dependence relation for $x_1$, \dots, $x_k$ over $K$, as required.
\end{proof}

We turn to making an immediate application of Theorem \ref{lineardependence}.
We assume as before that $L$ is a cyclic Galois extension of $K$ of degree $n$ and $\sigma$ generates
the Galois group of $L$ over $K$. Let 
\[
w(t)=\sum_{i=0}^k  b_it^i
\]
be a polynomial in $L[t]$. We may define the $K$-linear transformation $w(\sigma)$ acting
on $L$ in an obvious way. The set of elements $x\in L$ satisfying $w(\sigma)x=0$ clearly
is a $K$-subspace of $L$, whose dimension we will estimate using Theorem \ref{lineardependence}.

\begin{theorem} \label{dimension} Let $L$ be a cyclic Galois extension of $K$ of degree $n$ and suppose that $\sigma$ generates
the Galois group of $L$ over $K$. Let $k$ be an integer satisfying $1\le k\le n$ and let $w$ be a polynomial of degree $k$ in $L[t]$. Let 
\[
R=\{\,x\in L: w(\sigma)x=0\,\}.
\]
Then we have 
\[
\dim R\le k.
\]
\end{theorem}

\begin{proof} Suppose if possible that $\dim R>k$ and let $x_1$, \dots, $x_{k+1}$ be elements of $R$
that are linearly independent over $K$. Let
\[
w=\sum_{i=0}^k  b_it^i,
\]
where $b_k\ne 0$. Then we have
\[
b_0 x_i +b_1\sigma(x_i)+\cdots +b_k\sigma^k(x_i)=0,\quad 1\le i\le k+1.
\]
This is a homogeneous system of $k+1$ linear equations with $k+1\times k+1$ coefficient matrix $A$, say, where the $i,j$-entry of $A$ is
\[
\sigma^{j-1}(x_i), \quad 1\le i,j \le k+1.
\]
We deduce that $\det A=0$, since the system has a non-trivial solution. It follows that $\det A'=0$ also, where $A'$ is the transpose of $A$. But $A'$ is the $k+1\times k+1$ version of the matrix described in
Theorem \ref{lineardependence}. This theorem implies that $x_1$, \dots, $x_{k+1}$ are linearly
dependent over $K$, a contradiction. We conclude that $\dim R\le k$, as claimed.
\end{proof}

We can use the conclusion of Theorem \ref{lineardependence} to show that any $k$-dimensional
subspace of $L$ arises as a subspace $R$, as described in Theorem \ref{dimension}.

\begin{corollary} Let $L$ be a cyclic Galois extension of $K$ of degree $n$ and suppose that $\sigma$ generates
the Galois group of $L$ over $K$. Let $k$ be an integer satisfying $1\le k\le n$ and let $U$ be a subspace of $L$ of dimension $k$. Then there exists a polynomial $w$, say, of degree $k$ in $L[t]$ such that 
\[
U=\{\,x\in L: w(\sigma)x=0\,\}.
\]
\end{corollary}

\begin{proof} Let $x_1$, \dots, $x_k$ be a basis for $U$ over $K$ and let $S$ be the $k\times k$ matrix described in Theorem \ref{lineardependence}. $S$ is invertible by what we have proved and therefore its transpose $S'$ is also invertible. Let $\Sigma$ be the $k\times 1$ column vector whose $i$-th entry is
$\sigma^k(x_i)$, $1\le i\le k$. Then since $S'$ is invertible, there is a unique $k\times 1$ column vector
$B$, say, with
\[
S'B=\Sigma.
\]
If the entries of $B$ are $b_0$, \dots, $b_{k-1}$, we find that 
\[ 
w=t^k-b_{k-1}t^{k-1}-\cdots-b_1t-b_0
\]
is a polynomial with the required property. 
\end{proof}

Theorem 3.52 of \cite{Lidl} is a version of this result for finite fields.

We continue with the hypothesis that $L$ is a cyclic Galois extension of $K$ and $\sigma$ generates
the Galois group $G$ of $L$ over $K$. We suppose initially that $n=2m+1$ is odd. Enumerating the elements of $G$
as
\[
1,\,\sigma,\,\sigma^{-1}, \ldots,\,\sigma^m,\,\sigma^{-m}.
\]
we define, as before,  the  subspace $A^i$ of $\Alt(L)$ by
\[
A^i=\{\, f_{b,\,\sigma^i}:\, b\in L\,\}, \quad 1\le i\le m.
\]

\begin{theorem} \label{odddirectsumrank}  Let $L$ be a cyclic Galois extension of $K$ of odd degree $2m+1$ and suppose that $\sigma$ generates
the Galois group of $L$ over $K$. Define the subspaces $A^i$
of $\Alt(L)$ as above for $1\le i\le m$. Then for any integer $k$ satisfying $1\le k\le m$, all non-zero
elements of the $nk$-dimensional subspace
\[
A^1\oplus\cdots\oplus A^k
\]
of $\Alt(L)$ have rank at least $n-2k+1$.
\end{theorem}

\begin{proof} A typical element $g$, say, of $A^1\oplus\cdots\oplus A^k$ is expressible as
\[
g=\sum_{i=1}^k f_{b_i,\,\sigma^i}
\]
for suitable $b_1$, \dots, $b_k$ in $L$. Let $R$ denote the radical of $g$ and let $x$ be an element
of $R$. Then we have
\[
\sum_{i=1}^k \Tr(b_i(x\sigma^i(y)-\sigma^i(x)y))=0
\]
for all $y\in L$. The argument of Theorem \ref{isomorphism} implies that
\[
\sum_{i=1}^k \sigma^{-i}(b_ix)-b_i\sigma^i(x)=0.
\]
Applying $\sigma^k$ to each side, we obtain
\[
\sum_{i=1}^k \sigma^{k-i}(b_ix)-\sigma^k(b_i)\sigma^{i+k}(x)=0.
\]
We obtain
\[
w(\sigma)x=0,
\]
where $w\in L[t]$ is the polynomial 
\[
\sum_{i=1}^k \sigma^{k-i}(b_i)t^{k-i}-\sigma^k(b_i)t^{i+k}.
\]
Since this polynomial has degree at most $2k$, Theorem \ref{dimension} implies that $\dim R\le 2k$.
However, $g$ has even rank 
\[
n-\dim R
\]
and hence we cannot have $\dim R=2k$, since $L$ has odd dimension over $K$. We deduce that
$\dim R\le 2k-1$ and $g$ has rank at least $n-2k+1$.
\end{proof}

Theorem \ref{odddirectsumrank} was obtained by Delsarte and Goethals when $K$ is a finite field, although the methods they used involved counting and did not appeal to principles such as those
contained in Theorem \ref{dimension}. They in fact proved more, since they showed that
\[
A^1\oplus\cdots\oplus A^k
\]
contains elements of all possible even ranks lying between $n-1$ and $n-2k+1$. Furthermore, they also
showed that the dimension of any subspace of $\Alt(L)$ whose non-zero elements all have rank at least $n-2k+1$
is at most $nk$. Thus $A^1\oplus\cdots\oplus A^k$ is a subspace of maximal dimension with respect to this rank
property in the finite field case (subject to the proviso that $[L:K]$ is odd). See \cite{DG}, Theorems 4 and 7.
In the case of an infinite field $K$,  we do not know if $A^1\oplus\cdots\oplus A^k$ is a subspace 
of maximal dimension with respect to its rank
property.  We do not even know if $n$ is the largest dimension of a subspace of $\Alt(L)$
all of whose non-zero elements have rank $n-1$, although we conjecture this may be the case. We note also that, according to Theorem \ref{odddirectsumrank},  the subspace 
\[
A^1\oplus\cdots\oplus A^{m-1}
\]
has the 
property that all its non-zero elements have rank at least 4, and we have proved in Theorem 9 of [GQ] that
\[
\dim (A^1\oplus\cdots\oplus A^{m-1})=(m-1)n
\]
is the largest possible dimension for a subspace of this kind in $\Alt(L)$. Thus this subspace is certainly of maximal
dimension with respect to this rank property.

We can obtain a version of Theorem \ref{odddirectsumrank} for cyclic extensions of even degree.
The result is slightly less sharp. We omit the proof, as it is in the spirit of Theorem \ref{odddirectsumrank}.

\begin{theorem} \label{evendirectsumrank}  Let $L$ be a cyclic Galois extension of $K$ of even degree $2m+2$ and suppose that $\sigma$ generates
the Galois group of $L$ over $K$. 
Enumerate the elements of $G$ as 
\[
1,\,\sigma,\,\sigma^{-1}, \ldots,\,\sigma^m,\,\sigma^{-m},\,\sigma^{m+1}.
\]
Define the subspaces $A^i$
of $\Alt(L)$ as previously for $1\le i\le m+1$. Then for any integer $k$ satisfying $1\le k\le m$, all non-zero
elements of the $nk$-dimensional subspace
\[
A^1\oplus\cdots\oplus A^k
\]
of $\Alt(L)$ have rank at least $n-2k$.
\end{theorem}

Most of our findings in this paper admit interpretations in terms of the space of $n\times n$ skew-symmetric matrices with entries in a field $K$ admitting a Galois extension (more especially a cyclic
Galois extension) of degree $n$. Here, for example, is a version of  Theorem \ref{odddirectsumrank}.

\begin{theorem} \label{skewsymmetric} Let $K$ be a field admitting a cyclic Galois extension of odd degree $n=2m+1$ and let $\Alt_n(K)$ denote the vector space of all $n\times n$ skew-symmetric matrices with entries in $K$. There are $m$ subspaces $S^1$, \dots, $S^m$, say, each of dimension $n$ such that
\[
\Alt_n(K)=S^1\oplus \cdots \oplus S^m.
\]
The non-zero elements of each subspace $S^i$ have constant rank, depending on the order of the element in the Galois group used to define $S^i$. Furthermore,
 for any integer $k$ satisfying $1\le k\le m$, all non-zero
elements of the $nk$-dimensional subspace
\[
S^1\oplus\cdots\oplus S^k
\]
of $\Alt_n(K)$ have rank at least $n-2k+1$.

\end{theorem}

There is of course a version of Theorem \ref{evendirectsumrank} which may be expressed in terms
of skew-symmetric matrices. We omit the formal enunciation. We note that, for example,
any algebraic number field admits a cyclic extension of any specified positive degree. See \cite{Go},
Lemma 4. 

\section{Final observations}

\noindent Our construction of subspaces of alternating bilinear forms having, for example, the constant rank property, is clearly dependent on the existence of extensions of the underlying field, and entirely different results hold for algebraically closed fields such as the complex numbers $\mathbb{C}$.  For example, the following
is proved in \cite{Gr}, Corollary 1.2. Let $U$ be a subspace of  $\Alt_n(\mathbb{C})$ all of whose
non-zero elements have constant rank $r$. Then we have 
\[
\dim U\le  2(n-r)+1.
\]
In particular, if $n$ is odd and the non-zero elements of $U$ have constant rank $n-1$, then $\dim U\le 3$. When $\mathbb{F}$ is either an algebraically closed field of characteristic different from 2 or the field
of real numbers,  3-dimensional subspaces of $\Alt_5(\mathbb{F})$ where all non-zero elements
have rank $4$ are discussed in Corollaries 1 and 2 of \cite{Go}. We note that the paper \cite{I}
has a good discussion of literature relating to subspaces of matrices of constant or bounded rank. 

We have concentrated in this paper on the construction of alternating bilinear forms by Galois theoretic
methods. More generally, we can generate arbitrary bilinear forms, including symmetric ones,
by a simple extension of the technique we have already adopted. We briefly
sketch some details. 

Let $L$ be a Galois extension of $K$ of degree $n$ and let $G$ be the Galois group of $L$ over $K$.
Let $\Bil(L)$ denote the vector space of $K$-valued bilinear forms on $L\times L$. Let
\[
\sigma_1, \ldots, \sigma_n
\]
be the elements of $G$. Given elements $b_1$, \dots, $b_n$ of $L$, define
\[
f=f_{b_1, \ldots, b_n}: L\times L\to K
\]
by
\[
f(x,y)=\Tr((b_1\sigma_1(x)+\cdots +b_n\sigma_n(x))y)
\]
for all $x$ and $y$ in $L$. It is elementary to see that $f\in\Bil(L)$ and Artin's theorem shows that we obtain all of $\Bil(L)$ in this way. Furthermore, there is  a direct sum decomposition 
\[
\Bil(L)=B_1\oplus \cdots \oplus B_n,
\]
where each subspace $B_i$ has dimension $n$ and all its non-zero elements are invertible.

When $G$ is cyclic with generator $\sigma$, and we take $\sigma_i=\sigma^i$ for $1\le i\le n$, the argument of Theorem 6 shows that the non-zero elements of the $nk$-dimensional subspace 
\[
B_1\oplus \cdots \oplus B_k
\]
have rank at least $n-k+1$, for $1\le k\le n$. Section 6 of \cite{D} contains a statement and proof
of the assertions above when $K$ is finite.
\smallskip

\noindent {\bf Acknowledgements.  }  The authors wish to thank Gary McGuire and Tom Laffey for helpful
conversations which influenced the way in which the ideas in this paper were developed.

\end{document}